\documentclass[11pt]{article}

\usepackage{tikz}
\usepackage{subfigure}
\usepackage[english]{babel}

\usepackage[center]{caption2}
\usepackage{amsfonts,amssymb,amsmath,latexsym,amsthm}
\usepackage{multirow}

\newtheorem{thm}{Theorem}

\newtheorem{lem}[thm]{Lemma}

\newtheorem{conj}[thm]{Conjecture}

\begin{document}

\title{The edge spectrum of $K_4^-$-saturated graphs
\thanks{The work was supported by NNSF of China (No. 11671376) and  NSF of Anhui Province (No. 1708085MA18).}
}
\author{Jun Gao$^a$, \quad Xinmin Hou$^b$,\quad Yue Ma$^c$\\
\small $^{a,b,c}$ Key Laboratory of Wu Wen-Tsun Mathematics\\
\small School of Mathematical Sciences\\
\small University of Science and Technology of China\\
\small Hefei, Anhui 230026, China.
}

\date{}

\maketitle

\begin{abstract}
Given graphs $G$ and $H$, $G$ is $H$-saturated if $G$ does not contain a copy of $H$ but the addition of any edge $e\notin E(G)$ creates at least one copy of $H$ within $G$. The edge spectrum of $H$ is  the set of all possible sizes of an $H$-saturated graph on $n$ vertices. Let $K_4^-$ be a graph obtained from $K_4$ by deleting an edge. In this note, we show that (a) if $G$ is a $K_4^-$-saturated graph with $|V(G)|=n$ and $|E(G)|>\lfloor \frac{n-1}{2} \rfloor \lceil \frac{n-1}{2} \rceil +2$, then $G$ must be a bipartite graph;
(b) there exists a $K_4^-$-saturated non-bipartite graph on $n\ge 10$ vertices with size being in the interval  $\left[3n-11, \left\lfloor \frac{n-1}{2} \right\rfloor \left\lceil \frac{n-1}{2} \right\rceil +2\right]$.
Together with a result of Fuller and Gould in [{\it On ($\hbox{K}_t-e$)-Saturated Graphs. Graphs   Combin., 2018}], we determine the edge spectrum of $K_4^-$ completely, and a conjecture proposed by Fuller and Gould in the same paper also has been resolved.
\end{abstract}

\section{Introduction}
Given a graph $H$, we say a graph $G$ is {\it $H$-saturated} if $G$ does not contain a copy of $H$ but the addition of any edge $e\notin E(G)$ creates at least one copy of $H$ within $G$. The minimum (resp. maximum) number of edges of an $H$-saturated graph on $n$ vertices is known as the {\it saturation }(resp. {\it Tur\'an) number}, and denoted by $sat(n,H)$ (resp. $ex(n, H)$). A natural question is to determine all possible values $m$ between $sat(n, H)$ and $ex(n,H)$ such that there is an $H$-saturated graph whose size equals $m$. We call the set of all possible sizes of an $H$-saturated graph on $n$ vertices the {\it edge spectrum} of $H$ and denoted by $ES(n, H)$.
Clearly, for $m\in ES(n, H)$, we have $sat(n, H)\le m\le ex(n, H)$.

Write $K_n$ for a complete graph on $n$ vertices and $K_n^-$ for a graph obtained from $K_n$ by deleting one edge. Write $[m, n]$ for the set $\{m, m+1, \ldots, n\}$.

For complete graphs, Barefoot et al.~\cite{K_3} determined $ES(n, K_3)$;
Amin, Faudree, and Gould~\cite{K_4} evaluated $ES(n, K_4)$, and, more generally, $ES(n, K_p)$ for $p\ge 3$ was studied and given by Amin et al. in~\cite{K_p}. { Continuing the work, Gould et al. \cite{small-path} found the edge spectrum of small paths.} Recently, Faudree
et al. completely determined the edge spectrum of stars and partially gave the edge spectrum of paths in~\cite{Paths}.

For $K_t^-$, there is no $ES(n, K_t^-)$ for $t\ge 4$ has been completely determined so far.
It is well known that $ex(n, K_4^-)=\lceil\frac n2\rceil\lfloor\frac n2\rfloor$ and the upper bound can be realized by the complete bipartite graph $K_{\lceil\frac n2\rceil,\lfloor\frac n2\rfloor}$. Chen, Faudree, and Gould~\cite{Chen-Faudree-Gould-08} determined that $sat(n, K_4^-)=\left\lfloor\frac{3(n-1)}2\right\rfloor$ and the lower bound can be realized by the graph obtained from $K_{1,n-1}$ by adding $\lfloor\frac{n-1}2\rfloor$ independent edges.
In \cite{K_t-e}, Fuller and Gould  proved that
\begin{equation}\label{EQN: e1}
\left\{\left\lfloor\frac{3(n-1)}2\right\rfloor\right\}\cup \left[2n-4, \left\lfloor\frac n2\right\rfloor\left\lceil\frac n2\right\rceil-n+6\right]\subseteq ES(n, K_4^-)
\end{equation}
and proposed the following conjecture.
\begin{conj}[Fuller, Gould \cite{K_t-e}]\label{CONJ: c1}
The $K_4^-$-saturated graphs with sizes in the interval $\left[ \lfloor \frac{n}{2} \rfloor \lceil \frac{n}{2} \rceil -n+7,\lfloor \frac{n}{2} \rfloor \lceil \frac{n}{2} \rceil\right] $ are of two types: complete bipartite graphs with partite sets of nearly equal size, and 3-partite graphs with two partite sets of nearly equal size and one partite set of order one.
\end{conj}

In this paper, we first show that Conjecture~\ref{CONJ: c1} is true when the $K_4^-$-saturated graphs with sizes in the interval $\left[ \lfloor \frac{n-1}{2} \rfloor \lceil \frac{n-1}{2} \rceil +3, \lfloor \frac{n}{2} \rfloor \lceil \frac{n}{2} \rceil\right]$ and we also give $K_4^-$-saturated non-bipartite graphs with sizes in the interval $\left[3n-11,  \lfloor \frac{n-1}{2} \rfloor \lceil \frac{n-1}{2} \rceil +2\right]$. Combining with (\ref{EQN: e1}) we completely determine the edge spectrum of $K_4^-$.
Specifically, we show the following theorem.

\begin{thm}\label{THM: Main}
(a) If $G$ is a $K_4^-$-saturated graph with $|V(G)|=n$ and $|E(G)|>\lfloor \frac{n-1}{2} \rfloor \lceil \frac{n-1}{2} \rceil +2$, then $G$ must be a bipartite graph.

(b)  There exists a $K_4^-$-saturated non-bipartite graph on $n\ge 10$ vertices and $m$ edges where $3n-11\le m\le \lfloor \frac{n-1}{2} \rfloor \lceil \frac{n-1}{2} \rceil +2$.

(c) For $n\ge 10$,
$$
ES(n, K_4^-)=\left\{\left\lfloor \frac{3(n-1)}{2} \right\rfloor \right\}  \cup \left[2n-4 , \left\lfloor \frac{n-1}{2} \right\rfloor \left\lceil \frac{n-1}{2} \right\rceil +2\right] \cup \left\{ i(n-i) : i\in [1, n-1]\right\}.$$
\end{thm}

Clearly, (c) is a direct corollary from (\ref{EQN: e1}), (a) and (b). We give the proof of (a) and (b) in Section 2.

\section{Proof of Theorem~\ref{THM: Main}}

We first prove (a) of Theorem~\ref{THM: Main}.

\noindent{\bf Proof of Theorem~\ref{THM: Main} (a):}
Suppose to the contrary that $G$ is non-bipartite. 	Let $C$ be a shortest odd cycle in $G$ and $G'=G-V(C)$. Assume $|V(C)|=2t+1$ for some integer $t$.
Since $C$ is a shortest odd cycle and $G$ is $K_4^-$-free, we have $t\ge 1$, and for any $v \in V(G')$, $e(v, V(C))\le t$.
So
\begin{eqnarray*}e(G) &=&e(C)+e(G')+e(V(G'),V(C))\\
                      &\le& 2t+1+\left\lfloor \frac{n-2t-1}{2} \right\rfloor \left\lceil \frac{n-2t-1}{2} \right\rceil+(n-2t-1)t\\
                      &=&\left\lfloor \frac{n-1}{2} \right\rfloor \left\lceil \frac{n-1}{2} \right\rceil +2-(t-1)^2\\
                      &\le& \left\lfloor \frac{n-1}{2} \right\rfloor \left\lceil \frac{n-1}{2} \right\rceil +2,
\end{eqnarray*}
a contradiction. \medskip\rule{1mm}{2mm}
	

To show (b) of Theorem~\ref{THM: Main}, we first construct a family of $K_4^-$-saturated non-bipartite graphs. We write $K(X, Y)$ for the complete bipartite graph with partite sets $X$ and $Y$.


\noindent{\bf Construction A:}   Given nonnegative integers $n, a, b$ with $n\ge a+b+5$ and sets $I, A_1, A_2, B_1,$ $B_2, C$ with $|I|=1, |A_1|=|B_1|=2, |A_2|=a, |B_2|=b$ and $|C|=n-a-b-5$, let $M$ be  two independent edges connecting $A_1$ and $B_1$. Define $F_n(a,b)$ be the graph with vertex set
 $$V(F_n(a,b))=I\cup A_1 \cup A_2 \cup B_1 \cup B_2\cup C$$
 and edge set
 $$E(F_n(a,b))=E(K( A_1\cup A_2\cup C, B_2))\cup E(K(A_2, B_1))\cup E(K(I,  A_1\cup B_1\cup C))\cup M.$$

We can count the number of edges of $F_n(a,b)$ directly from the construction.
\begin{lem}\label{LEM: l0}
	$e(F_n(a, b))=b(n-b-3)+n+a-b+1$.
\end{lem}

\begin{lem}\label{LEM: l1}
If $b\ge 2$, then $F_n(a,b)$ is $K_4^-$-saturated for any $a\ge 0$.
\end{lem}
\begin{proof}
Denote $I=\{x\}$, $A_1=\{u_1, u_2\}$, and $B_1=\{v_1, v_2\}$. Assume $u_iv_i \in E(F_n(a,b))$ for $i=1,2$.
Clearly, $F_n(a,b)$ is $K_4^-$-free since $F_n(a,b)$ has only two triangles $xu_1v_1$ and $xu_2v_2$ sharing a common vertex $x$.
In the following, we show that $F_n(a,b)$ is $K_4^-$-saturated. Choose any edge $e=t_1t_2\notin E(F_n(a,b))$, we show that $F_n(a,b)+e$ contains a copy of $K_4^-$.

	(i) $t_1, t_2\in A_1\cup A_2 \cup C$. Choose $b_1, b_2 \in B_2$ (this can be done since $|B_2|=b\ge 2$). Then $\{t_1,t_2,b_1,b_2\}$ induces a copy of  $K_4^-$ in $F_n(a,b)+e$.

	(ii) $t_1, t_2\in B_2$. Then $\{t_1,t_2, u_1, u_2\}$ induces a $K_4^-$ in $F_n(a,b)+e$.

	(iii) $t_1, t_2\in B_1$. Then $\{t_1,t_2, u_1, x\}$ induces a $K_4^-$ in $F_n(a,b)+e$.

	(iv) $t_1\in B_1$ and $t_2 \in B_2$. Assume $t_1=v_1$. Then $\{t_1,t_2,x,u_1\}$ induces  a $K_4^-$ in $F_n(a,b)+e$.

	(v) $t_1=x, t_2\in A_2$. Then $\{x,t_2, v_1, v_2\}$ induces a $K_4^-$ in $F_n(a,b)+e$. Similar argument for the case  $t_1=x, t_2 \in B_2$.

	(vi) $t_1=u_i, t_2=v_{3-i}$ for $i=1,2$. Then $\{x, u_i, v_1, v_2\}$ induces a $K_4^-$ in $F_n(a,b)+e$.

    (vii) $t_1\in C, t_2\in B_1$. Assume $t_2=v_1$. Then $\{x, t_1, v_1, u_1\}$ induces a $K_4^-$ in $F_n(a,b)+e$.

   This completes the proof.

\end{proof}

Now, define $f_n(a,b)=|E(F_n(a,b))|$ and $$\mathcal{F}_n(a,b)=\{f_n(a,b) : b\ge 2, a\ge 0\}.$$ By Lemma~\ref{LEM: l1}, $\mathcal{F}_n(a,b)\subseteq ES(n,K_4^-)$.

\begin{lem}\label{LEM: l2}
$$\left[3n-11, \left\lfloor \frac{n-1}{2} \right\rfloor \left\lceil \frac{n-1}{2} \right\rceil +2\right]\subseteq \mathcal{F}_n(a,b).$$

\end{lem}
\begin{proof}
By Lemma~\ref{LEM: l0}, $$f_n(a, b)=b(n-b-3)+n+a-b+1.$$
Clearly, for fixed $a$, $f_n(a, b)$ is a concave function of $b$ with maximum $f_n(a, \lfloor (n-5)/2 \rfloor)$, and for fixed $b$, $f_n(a,b)$ is an increasing function of $a$. Note that $n-5-b \ge a\ge 0$. We have
$$\mathcal{F}_n(a,b)=\bigcup_{b=2}^{\lfloor (n-5)/2 \rfloor}\left[f_n(0,b), f_n(n-b-5,b)\right].$$
We claim that $\mathcal{F}_n(a,b)$ is an interval. To show the claim, it is sufficient to check that two consecutive intervals are overlap.
In fact,
\begin{equation*}
\begin{split}
f_n(0,b+1)&=(b+1)(n-b-1-3)+n-(b+1)+1\\
&=b(n-3-b)+2n-3b-4\\
&\le f_n(n-b-5, b).
\end{split}
\end{equation*}
Note that $f_n(0,2)=3n-11$ and $f_n(n-\lfloor (n-5)/2 \rfloor-5, \lfloor (n-5)/2 \rfloor)=\lfloor \frac{n-1}{2} \rfloor \lceil \frac{n-1}{2} \rceil +2$.
We have
$$\left[3n-11, \left\lfloor \frac{n-1}{2} \right\rfloor \left\lceil \frac{n-1}{2} \right\rceil +2\right]\subseteq \mathcal{F}_n(a,b).$$
\end{proof}


(b) of Theorem~\ref{THM: Main} follows directly from Lemmas~\ref{LEM: l1} and~\ref{LEM: l2}.

\end{document}